\documentclass[reqno]{amsart}
\usepackage{amssymb}
\usepackage[dvips]{epsfig}
\usepackage{graphicx}
\usepackage{color}
\usepackage{amsmath}
\usepackage{amssymb}
\usepackage{amsfonts}
\usepackage{amsthm}
\usepackage{bbm}
\usepackage{tikz}
\usepackage{bm}
\usepackage[colorlinks=true,pdfstartview=FitV,linkcolor=magenta,citecolor=cyan]{hyperref}
\usepackage{mathtools}

\usepackage{stmaryrd}

\newcommand{\Z}{\mathbb Z}

\newtheorem{thm}{Theorem}[section]
\newtheorem{lem}[thm]{Lemma}

\theoremstyle{remark}
\newtheorem{rem}{\bf Remark}[section]
\theoremstyle{definition}
\newtheorem{defn}[thm]{Definition}

\newtheorem{Problem}[thm]{\bf Problem}
\numberwithin{equation}{section}

\begin{document}
\title[Localization for coupled harmonic oscillators]{Localization for  random coupled harmonic oscillators on $\Z^d$}

\author[H. Cong]{Hongzi Cong}
\address[H. Cong]{School of Mathematical Sciences, Dalian University of Technology, Dalian 116024, China}
\email{conghongzi@dlut.edu.cn}


\author[Y. Shi]{Yunfeng Shi}
\address[Y. Shi] {School of Mathematics,
Sichuan University,
Chengdu 610064,
China}
\email{yunfengshi@scu.edu.cn}

\author[Z. Zhang]{Zhihan Zhang}
\address[Z. Zhang]{School of Mathematical Sciences, Dalian University of Technology, Dalian 116024, China}
\email{dlutzzh@163.com}
\date{\today}



\begin{abstract}
In this paper we consider the localization properties of coupled harmonic oscillators in  random media. Each of these oscillators is restricted  to the lattice $\Z^d$.  We show that for most  states and an arbitrary choice  of the random media,  the long time localization for the coupled system holds  in a time scale larger than the polynomial one.  
\end{abstract}

\maketitle

\section{Introduction and main results}
Localization  for quantum particles in  random media is of fundamental importance in both physics and mathematics since the remarkable work of Anderson \cite{And58}.  While considerable progress has been made in dealing with the linear disordered models, much less  is known for the nonlinear one. In fact,  the nonlinear random models   appear   in a wide class of physical systems,   such as the classical optics \cite{SBFS07} and Bose-Einstein condensate \cite{Sha07} in the presence of disorder. 
In this paper we consider the localization properties of coupled harmonic oscillators in a random media. More precisely, we study the Hamiltonian of the form
\begin{equation}\label{122301}
H(q,\bar q)=H_2(q,\bar q)+H_4(q,\bar q)+R(q,\bar q),
\end{equation}
where
\begin{align*}
H_2(q,\bar q)=\sum_{\bm j\in\mathbb{Z}^d}v_{\bm j}\left|q_{\bm j}\right|^2,\ H_4(q,\bar q)=\frac12\sum_{\bm j\in\mathbb{Z}^d}\left|q_{\bm j}\right|^4
\end{align*}
and
\begin{align*}
R(q,\bar q )=O(q^6).
\end{align*}
Here we assume $v=\{v_{\bm j } \}_{\bm j\in\Z^d}$ is a family of independent identically distributed (i.i.d) random variables in $[0,1]$ and $R$ only contains the terms of {\it short-range} action, which will be defined as follows. 
For each $\bm \alpha\in\mathbb{N}^{\mathbb{Z}^d}$,  define 
\begin{align*}
{\rm supp}\  \bm \alpha
&=\left\{\bm j\in\mathbb{Z}^d:\ \alpha_{\bm j}\neq 0\right\}, \ \Delta( \bm \alpha)
=
\sup_{\bm j,\bm j'\in{\rm supp}\ \bm \alpha}
 \left|\bm j-\bm {j}'\right|_1,\ 
| \bm \alpha|=\sum_{\bm j\in \mathbb{Z}^d}\left|\alpha_{\bm j}\right|,
\end{align*}
where
$$|\bm j|_1=\sum_{1\leq i\leq d}\left|j_i\right|\ {\rm for}\ \bm j=
\left( j_{1}, j_{2},\cdots, j_{d}\right)\in \mathbb{Z}^d.$$
We assume that $R$ has the form of 
\begin{equation}\label{010901} 
	R(q,\bar q) =
         \sum_{\bm\beta,\bm\gamma\in\mathbb{N}^{\mathbb{Z}^d},|\bm \beta+\bm \gamma|=6\atop
         \Delta( \bm \beta+\bm\gamma)\leq 1}
         R^{\bm \beta \bm \gamma}
         q^{\bm \beta}\bar q^{\bm \gamma}
         \end{equation}
     with $ \left| R^{\bm \beta \bm \gamma} \right| \le 1$.
   Our main goal is to prove that {\it for all} $v
  	$ and ``most" of solutions $q(t)$ with small initial states (of size $\epsilon$) of \eqref{122301}  
  remain localized in a time scale of 
  \begin{equation*}
  	|t|\leq \exp \left\{\frac{c|\ln \epsilon|^2}{\ln\ln \epsilon^{-1}}\right\},\quad c>0
  \end{equation*}
  in the phase space
  \begin{align*}
  	\ell_{\infty}^{\sigma}=\left\{q=(q_{\bm j})_{\bm j\in\mathbb{Z}^d}:\left\|q\right\|_\sigma:=\sup_{\bm j\in\mathbb{Z}^d}\left|q_{\bm j}\right|\left(1+\left|\bm j\right|_1\right)^{\sigma}<\infty\right\}
  \end{align*}
  with $\sigma\geq d+1$ (cf. Theorem \ref{main} in the following for details). 
 To clarify our main motivation, we first recall some results about the celebrated  discrete  nonlinear Schr\"odinger equations (DNLS) with random potentials. So  consider 
 \begin{align}\label{DNLS}
\sqrt{-1} \frac{dq_{\bm j}}{dt}=\epsilon(\bm\Delta q)_{\bm j}+v_{\bm j} q_{\bm j}+\delta |q_{\bm j}|^2q_{\bm j},
 \end{align}
 where the discrete Laplacian $\bm\Delta$ is defined as $(\bm \Delta q)_{\bm j}=\sum\limits_{|\bm e|_1=1}q_{\bm j+\bm e}.$
Note that the Hamiltonian associated with \eqref{DNLS} is given by 
\begin{align}\label{HNLS}
H_{NLS}=H_2+\frac\delta2 H_4+ \epsilon \sum_{\bm i\in\Z^d}\sum_{\bm |\bm i-\bm j|_1=1}q_{\bm i}\bar q_{\bm j},
\end{align}
where $H_2$ and $H_4$ are defined in \eqref{122301}. The first rigorous localization result   concerning  \eqref{HNLS} was obtained by Fr\"ohlich-Spencer-Wayne \cite{FSW86}  with the general $H_4$ (which means that there exists nonresonant terms in degree 4), where the full dimensional KAM tori were constructed in the case $\epsilon=0$ for most $v$. Later, Bourgain-Wang \cite{BW08} proved the existence of finitely dimensional KAM tori for \eqref{DNLS}
in the case of $0<\epsilon+\delta\ll1$ for most $v$. These KAM results allow an infinite time scale but require more restrictions on the initial states.  Another perspective of the research  is to work in finite time scales. Along this line, Benettin-Fr\"ohlich-Giorgilli \cite{BFG88} first proved a Nekhoroshev type theorem concerning the perturbations of $H_2$ with small coupled Hamiltonians of degree  at leas $3$ and of short-range interactions for most $v$ with Gauss distributions. In the important work \cite{WZ09}, the authors proved the first long-time nonlinear Anderson localization with roughly initial state (i.e. $q(0) \in\ell^2(\Z)$) in a polynomial long time scale in the presence of $\bm\Delta$. Recently, Cong-Shi-Zhang \cite{CSZ21} extended the time scale of \cite{WZ09} to the one obtained in \cite{BFG88}. We also refer to \cite{FKS09,FKS12} for more recent progress on the nonlinear Anderson localization.  We should remark that all the works mentioned above require essential restrictions on the random potentials $v$ when establishing the localization. In fact, in \cite{Bou04}, Bourgain raised the following problem
\begin{Problem}[cf. page 1348 in \cite{Bou04})]\label{prob}
{\it What may be said about the stability of $\{I_j=|q_j|^2\}_{j\in\Z}$ for the \eqref{DNLS} (with $d=1$) without any assumption on the $\{v_j \}_{j\in\Z}$ ?}
\end{Problem}
The present work tries to answer this problem in the case that  the discrete Laplacian part is replaced by  short-range interacting Hamiltonians of degree at least $5$. While we can not address the original problem of Bourgain,  our result holds for all lattice dimensions.  In this context, we also extend the work of \cite{BFG88}.  To answer the above Bourgain's problem, we need {\it sufficiently many}  parameters to overcome the resonances difficulty. 
Without any restriction on $v$ (i.e, the outer parameters),  one  has to choose the  initial sates to be the parameters (i.e., the inner parameters).  Since the initial states come from $\ell_\infty^{\sigma}$,
 the inner parameters must decay with $\bm j\in\Z^d,$  which is significantly different from the outer parameters case.  So  it is highly nontrivial to handle the resonances in this inner  parameters case.  In this paper we resolve this issue by introducing new non-resonant conditions. 


Now we define  the inner parameters. Let 
\begin{equation}\label{122302}
    \Omega=\left\{\zeta \in\mathbb{R}^{\mathbb{Z}^d}:\ 
    \zeta_{\bm j}
    \left(1 + \left|\bm j\right|_1\right)^{2\sigma}\in\left[0,1\right]\right\}.
    \end{equation}
For each $r>0$, denote 
\begin{equation*}
B_{\sigma}(r)=\left\{q\in\ell_{\infty}^\sigma: 
\left\|q\right\|_{\sigma}
\leq r\right\}.
\end{equation*}
Denote by ${\rm mes}(\cdot)$ the standard normalized product measure on $\Omega
$. Our main result is 
\begin{thm}\label{main}
	Let $\sigma\geq d+1$. Consider the Hamiltonian \eqref{122301} where $R$ is given by \eqref{010901}. For any $\eta>0$, there are $\epsilon_0=\epsilon_0(\eta,d,\sigma)$ and  $\Omega'\subset\Omega$ with ${\rm mes}(\Omega\setminus\Omega')\leq \eta$ so that for $\zeta\in\Omega'$ and $\epsilon<\epsilon_0$  the following holds true: 
There exists a symplectic map $\Phi:B_{\sigma}(\epsilon)\rightarrow B_{\sigma}(2\epsilon)$, which is close to the identity map, such that
\begin{align*}
H\circ \Phi\left(\widetilde q,\overline{\widetilde q}\right)=H_2+\mathcal{J}+Z+O\left(\exp\left\{\frac{\left|\ln \epsilon\right|^2}{ 10^4 \sigma \ln\ln\epsilon^{-1}}\right\}\right),
\end{align*}
where 
 \begin{equation*}
 \mathcal{J}=\frac12\sum_{\bm j\in\mathbb{Z}^d}\left(\left|\widetilde q_{\bm j}\right|^2-\zeta_{\bm j}\right)^2,
 \end{equation*}
and $Z=O(\|\widetilde q\|_\sigma^6)$ depends only  on  the action variable $\widetilde{I}=\left(\widetilde I_{\bm j}=|\widetilde q_{\bm j}|^2\right)_{\bm j\in\mathbb{Z}^d}$.  
Moreover,  for each initial state $q(0)$
satisfying $\left\|q(0)\right\|_\sigma\leq \epsilon$ and \begin{equation*}
	|\widetilde{q}_{\bm j}(0)|^2=\epsilon^2\zeta_{\bm j},
\end{equation*}
one has 
\begin{align*}
	\left|
	 \left| q_{\bm j}(t)  \right|^2 - \left| q_{\bm j}(0)  \right|^2
	\right|<\epsilon^{2}  (1+|\bm j|_1)^{-3\sigma}\ {\rm for}\ \forall\  \bm j \in \mathbb{Z}^d
\end{align*}
assuming 
\begin{equation*}
\left|t\right|\leq \exp\left\{\frac{\left|\ln \epsilon\right|^2}{10^4 \sigma \ln\ln\epsilon^{-1}}\right\}.
\end{equation*}
\end{thm}
\begin{rem}
The long time scale of our result is also a Nekhoroshev type one as in \cite{BFG88}. Our result also allows arbitrary random potentials $v=\{v_{\bm j}\}$. 
\end{rem}
If an additional restriction is imposed on the initial states, we can obtain the exponential long time localization result  as that of \cite{CMS22}.  More precisely, we have

\begin{thm}\label{th2}
	Let $\sigma\geq d+1$. 
	Consider the Hamiltonian \eqref{122301} where $R$ is given by \eqref{010901}.
    For each 
    \begin{equation}\label{epd}
    	0<\epsilon<2^{-12d-9},
    \end{equation}
     any $v$ and $q(0)$ satisfying 
\begin{equation}\label{cd2}
	\left\|q(0)\right\|_{\sigma}\leq\epsilon \quad \text{and} \quad q_{\bm 0}(0)=0, 
\end{equation}
one has
\begin{align}\label{1.2}
\left|
	 \left| q_{\bm j}(t)   \right|^2- \left| q_{\bm j}(0)   \right|^2
	\right|
	<\epsilon^{2}  (1+\langle\bm j\rangle)^{-3\sigma}, \quad  
     \forall\  \bm j \in \mathbb{Z}^d
\end{align}
for any
\begin{equation*}
|t|\leq \epsilon^{-3} 2^{\sigma},
\end{equation*}
where $\langle \bm j\rangle=\max\{\left|\bm j\right|_1,1\}.$
In particular, if $\sigma\geq \epsilon^{-1}$, the stability time is exponential long, namely,  \eqref{1.2} holds for
\begin{equation}\label{th2.2}
	|t| < \epsilon^{-3} 2^{ \epsilon^{-1}}.
\end{equation}
\end{thm}


\subsection*{The strategy of the proof }

    The proof of Theorem \ref{main} is based on the Birkhoff normal form technique   of Bourgain-Wang \cite{BW07} (cf.  Wang-Zhang \cite{WZ09}). 
         So we first rewrite the Hamiltonian of \eqref{122301} as 
    \begin{equation}\label{ss}
    \widehat{H}(q,\bar q):=\epsilon^{-4}H(\epsilon q,\epsilon\bar q).
    \end{equation}
   Then for $\zeta\in\Omega$,
   we introduce the notation 
    \begin{equation}\label{eJ}
    	\epsilon J_{\bm j} = \left|q_{\bm j}\right|^2 - \zeta_{\bm j},
    \end{equation}
    and  assume 
        \begin{equation}\label{J}
    	|J_{\bm j}| < (1+|\bm j|_1)^{-3\sigma}.
    \end{equation}
    As  a result,  the Hamiltonian $\widehat{H}(q,\bar q)$  becomes
    \begin{equation}\label{122401}
    \widehat{H}(q,\bar q)=\sum_{\bm j\in\mathbb{Z}^d}\omega_{\bm j}\left|q_{\bm j}\right|^2+\frac{\epsilon^{2}}2\sum_{\bm j\in\mathbb{Z}^d}J_{\bm j}^2+\epsilon^2R(q,\bar q),
    \end{equation}
    where
    \begin{equation}\label{122303}
    \omega_{\bm j}=\epsilon^{-2}v_{\bm j}+\zeta_{\bm j}.
    \end{equation}
    \begin{rem}
    	 The estimate (\ref{J}) is of vital importance, and it can be kept during the iterations.
    \end{rem}

Next, we introduce our new non-resonant conditions which are essential for our iterations. 
\begin{defn}
Given any  $\eta>0$ and any large $M>0$, we say the frequency $\bm \omega$ is $(\eta,M)$-non-resonant if for any $0\neq  \bm k\in\mathbb{Z}^{\mathbb{Z}^d}$ satisfying $\left|\bm k\right|$ and $\Delta(\bm k)\leq M$, one has
\begin{align}\label{122106}
\left|\sum_{\bm j\in\mathbb{Z}^d}k_{\bm j}\omega_{\bm j}\right|
>
\frac{\eta}{ (1+\bm k^-)^{3\sigma}
	\left(2+10\Delta(\bm k)\right)^{2d\left|\bm k\right|}
},
\end{align}
where
$$\bm k^-=\min_{\bm j \in \operatorname{supp} \bm k}
|\bm j|_1.$$
\end{defn}
\begin{rem}
	As compared to \cite{BFG88}, we  replace $d+1$  (cf. (3.14) in \cite{BFG88}) with $3\sigma$ in  our  non-resonant  conditions to guarantee  both the measure estimate and the iterations using Birkhoff normal form technique. This type of non-resonant conditions  plays an essential role in dealing with inner parameters.  Definitely, our non-resonant conditions make sense because we have a better estimate on  $J$ (cf. \eqref{J}).
	We also remark that the factor $(2+10\Delta(\bm k))^{2d|\bm k|}$ here can be regarded as a replacement of   certain geometric descriptions in \cite{BFG88}.
\end{rem}

Under the above preparations, 
by the standard Birkhoff normal form technique together with  the $(\eta,M)$-non-resonant conditions,
    one can obtain a normal form of higher order, i.e. there exists a symplectic map $\Phi$, which is close to the identity map, such that
    \begin{equation*}
    	\widehat{H}\circ \Phi(q,\bar q)
    	=\sum_{\bm j\in\mathbb{Z}^d}\omega_{\bm j}\left|q_{\bm j}\right|^2
    	+\frac{\epsilon^{2}}2\sum_{\bm j\in\mathbb{Z}^d}J_{\bm j}^2
    	+O(|I|) + O\left(\epsilon^{0.24M}\right),
    \end{equation*}
    where $O(|I|)$ are  the terms depending only on the  action variables $I=\left(I_{\bm j}\right)_{\bm j\in\mathbb{Z}^d}$ with $I_{\bm j}=\left|q_{\bm j}\right|^2$.  
    	Then by optimizing $M \sim \frac{\ln \epsilon^{-1}}{\ln \ln \epsilon^{-1}}$ (cf. (\ref{M}) for details) and through a standard bootstrip lemma we finish the proof of Theorem \ref{main}.
    
    
    	The proof of Theorem \ref{th2} follows from  a \textbf{tame} type inequality even without the Birkhoff normal form technique. Preciously, if $q\in\ell_\infty^\sigma$ and $q_{\bm 0}=0$ then  $q\in \widetilde{\ell}_\infty^\sigma$ (cf. (\ref{well})). Further in the  phase space
    	$\widetilde{\ell}_{\infty}^{\sigma}$, the following estimate holds (cf. (\ref{IH2}) for details)
    	\begin{equation*}
    		\left|  \{ I_{\bm j}, H \} \right|
    		\le \epsilon^{5} (1+\langle \bm j \rangle )^{-3\sigma}
    		2^{-\sigma}, 
    	\end{equation*}
        which finishes the proof of Theorem \ref{th2}.
        The key point here is that $\sigma$ can be chosen free from $\epsilon.$

\subsection*{Organization of the paper}
The structure of the Hamiltonian is studied in \S \ref{struh}. The Birkhoff normal form results are stated and proved in \S \ref{BKN}. The measure estimate concerning the non-resonant frequencies is established in \S \ref{Em}. The proofs of Theorem \ref{main} and \ref{th2} are given in \S \ref{mainthm}.

\section{Structure of the  Hamiltonian}\label{struh}

For any $
	\bm \alpha,
 \bm \beta,
 \bm \gamma\in \mathbb{N}^{\mathbb{Z}^d}
$ and $\zeta\in\Omega$, introduce the monomials in the form
\begin{align*}
J^{\bm \alpha}q^{\bm \beta}\bar q^{\bm \gamma}
=\prod_{\bm j\in\mathbb{Z}^d}
J_{\bm j}^{ \alpha_{\bm j}}
q_{\bm j}^{\beta_{\bm j}}
\bar q_{\bm j}^{\gamma_{\bm j}},
\end{align*}
where $J_{\bm j}$ is defined by (\ref{eJ}).
Let $\bm n=(\bm \alpha, \bm \beta,\bm \gamma)$
      and we define
\begin{align*}
{\rm supp}\ \bm n={\rm supp}\ (\bm \alpha+ \bm \beta+ \bm \gamma),\quad
\Delta(\bm n)
=\Delta(\bm \alpha+ \bm \beta+ \bm \gamma),
\quad |\bm n|=|2 \bm \alpha+  \bm \beta+ \bm \gamma|.
\end{align*}
Then the perturbation $\epsilon^2R(q,\bar q)$ in (\ref{122401}) is turned into
\begin{equation*}
\epsilon^2R(q,\bar q)=\sum_{|\bm \beta+\bm \gamma|=6, \Delta(\bm n)=1\atop |\bm \alpha|=0
}R(\bm n)J^{\bm \alpha}q^{\bm \beta}\bar q^{\bm \gamma}
\end{equation*}
\begin{defn} \label{defpb}
	Define $\bm m=(\widetilde{\bm \alpha}, \widetilde{\bm \beta} ,\widetilde{\bm \gamma}), 
	\bm \mu=(\widehat{\bm \alpha}, \widehat{\bm \beta} ,\widehat{\bm \gamma})
	\in\mathbb{N}^{\mathbb{Z}^d}
	\times\mathbb{N}^{\mathbb{Z}^d}
	\times\mathbb{N}^{\mathbb{Z}^d}$.
Then for the given two Hamiltonians
\begin{align*}
H(q,\bar{q})
&=
\sum_{\bm \mu
    }H(\bm \mu)
J^{\widehat{\bm \alpha}}q^{\widehat{\bm \beta}}\bar {q}^{\widehat{\bm \gamma}},\\
G(q,\bar{q})&=\sum_{\bm m
    }
G(\bm m)
J^{\widetilde{\bm \alpha}}q^{ \widetilde{\bm{\beta}}}
\bar {q}^{  \widetilde{\bm{\gamma}}}
\end{align*}
the Poisson bracket of $H$ and $G$ is defined by
\begin{align*}
\{H,G\}&=\sqrt{-1}
\sum_{\bm \mu,\bm m
    }
H(\bm \mu)G(\bm m) \cdot (\star),
\end{align*}
where
\begin{align*}
	(\star)=
&
\sum_{\bm j\in\mathbb{Z}^d}
 \left(
    \left(\star\star\right)+\left(\star\star\star\right)
\right)
\prod_{{\bm k}\neq {\bm j}}
J_{\bm k}^{ \widehat{\alpha}_{\bm k}+ \widetilde{\alpha}_{\bm k}}
q_{\bm k}^{\widehat{\beta}_{\bm k}+\widetilde\beta_{\bm k}}
\bar {q}_{\bm k}^{\widehat{\gamma}_{\bm k}+\widetilde\gamma_{\bm k}},
\end{align*}
\begin{equation}\label{bp1}
\left(\star\star\right)=  \epsilon^{-1}
    \left(
      \widehat{\alpha}_{\bm j}(\widetilde{\beta}_{\bm j}-\widetilde{\gamma}_{\bm j} ) +
      \widetilde{\alpha}_{\bm j} \left(\widehat{\gamma}_{\bm j}-\widehat{\beta}_{\bm j}\right)
    \right)
  J_{\bm j}^{\widehat{\alpha}_{\bm j}+\widetilde{\alpha}_{\bm j}-1}
  q_{\bm j}^{\widehat{\beta}_{\bm j}+\widetilde\beta_{\bm j}}
  \bar {q}_{\bm j}^{\widehat{\gamma}_{\bm j}+\widetilde\gamma_{\bm j}},
\end{equation}
and
\begin{equation}\label{bp2}
 \left(\star\star\star\right)=\left(\widehat\beta_{\bm j}\widetilde\gamma_{\bm j}-\widetilde\beta_{\bm j}\gamma_{\bm j}\right)
J_{\bm j}^{\widehat\alpha_{\bm j}+\widetilde{\alpha}_{\bm j}}
q_{\bm j}^{\widehat\beta_{\bm j}+\widetilde\beta_{\bm j}-1}
\bar {q}_{\bm j}^{\widehat\gamma_{\bm j}+\widetilde\gamma_{\bm j}-1}.
\end{equation}
\end{defn}
Using (\ref{bp1}), (\ref{bp2}) and following the proof of (3.16) in Wang-Zhang \cite{WZ09}, we can get the
estimate of the Poisson bracket.

\begin{lem}\label{pb}
	For the given Hamiltonians $H$ and $G$, the Poisson bracket of them can be estimated by
	\begin{align}\label{lem}
		\left|\{H,G\}(\bm n)\right|
		&\leq
		2
		\epsilon^{-1}
		\left(
		2^{|\bm n|}
		\left(\Delta(\bm \mu)+ \Delta(\bm m)\right)
		 \right)
		\left( \left| \bm n \right|+2 \right)^{2}
		|H(\bm \mu)| |G(\bm m)| .
	\end{align}
\end{lem}
\begin{rem}
	The estimate of the Poisson bracket in Wang-Zhang is given by
	\begin{align}
		\label{WZ}
		\left|\{H,G\}(\bm n)\right|
		&\leq
		\left(
		2^{|\bm n|}
		\left(\Delta(\bm \mu)+ \Delta(\bm m)\right)
		\right)
		\left( \left| \bm n \right|+2 \right)^{2}
		|H(\bm \mu)| |G(\bm m)| .
	\end{align}
    The factor $\epsilon^{-1}$ in the right hand of (\ref{lem}) comes from
	\begin{equation*}
		\frac{\partial J_{\bm j}}{ \partial \bar{q}_{\bm j}}
		=\frac{1}{\epsilon}
		\frac{\partial q_{\bm j} \bar{q}_{\bm j} }{ \partial \bar{q}_{\bm j}}
		=\epsilon^{-1} q_{\bm j}.
	\end{equation*}
	Particularly, when $\widehat{\bm \alpha} = \widetilde{\bm \alpha}=\bm 0$, the estimate (\ref{lem}) is the same as (\ref{WZ}), since (\ref{bp1}) vanishes.
\end{rem}
    
\begin{rem}\label{n+}
	We introduce some facts here.
	For the $\bm n ,\bm m$ and $\bm \mu$ given in the Lemma \ref{pb}, firstly we have 
	     \begin{equation}\label{nmum}
	 	    	|\bm n|=|\bm \mu|+|\bm m|-2,
	 	 \end{equation}
 	 \begin{equation}\label{Dn}
 	 	\Delta(\bm n)\leq \Delta(\bm \mu)+\Delta(\bm m),
 	 \end{equation}
 	 and
 	 \begin{equation*}
 	 |\bm{ \alpha}| = |\bm{ \widehat\alpha}| + |\bm{ \widetilde\alpha}| - 1
 	 \quad \mbox{or} \quad
 	 |\bm{ \alpha}| = |\bm{ \widehat\alpha}| + |\bm{ \widetilde\alpha}| .
 	 \end{equation*}
    Secondly, if we assume that 
    \begin{align}\label{Dm}
    	\Delta(\bm m)\le \frac{|\bm m|-2}{4} \quad \text{and} \quad
    	\Delta(\bm \mu)\le \frac{|\bm \mu|-2}{4},
    \end{align}
    then we have
    \begin{equation}\label{n}
    	\Delta(\bm n)
    	\le 
    	\frac{|\bm n|-2}{4},
    \end{equation}
    where we have used (\ref{nmum}), (\ref{Dn}) and (\ref{Dm}).
    Furthermore, define by
    \begin{equation*}
    	\bm n^+ = \max_{\bm j \in \operatorname{supp} \bm n} |\bm j|_1.
    \end{equation*}
    Then we have that
    \begin{equation}\label{m+}
    	 \bm m^+,\ \bm\mu^+ \le \bm n^+ + d (\Delta(\bm m) +\Delta(\bm \mu)).
    \end{equation}
\end{rem}

\section{The Birkhoff Normal Form}\label{BKN}

We now construct the symplectic transformation $\Gamma$ (by a finite-step  induction) in the spirit of Birkhoff normal form.

\subsection{The First Step}
	At the first step (i.e., $s=1$),  we let
	\begin{align*}
		\mathcal{H}_1:=\widehat{H}(q,\bar q),
	\end{align*}
where $\widehat{H}(q,\bar q)$ is given by (\ref{122401}).
    Then
	we  rewrite $\mathcal H_1$ as
\begin{equation*}
\mathcal{H}_1=D+\mathcal{J}+Z_1+R_1,
\end{equation*}
where
	\begin{align}
		\nonumber D&=\sum_{\bm j\in\mathbb{Z}^d}
		\omega_{\bm j}
		|q_{\bm j}|^2, \quad
			\mathcal{J}=\frac{\epsilon^{2}}2\sum_{\bm j\in\mathbb{Z}^d}J_{\bm j}^2,
		\\ \label{Z1}
		Z_1&=\sum_{
	        |\bm n|=6, \Delta(\bm n)=1
	        \atop
	       |\bm \alpha|=0, \bm \beta = \bm \gamma
	 }
		Z_1(\bm n)
		J^{\bm \alpha} q^{\bm \beta} \bar q^{\bm \gamma},\\
		\nonumber
		R_1&=\sum_{
			|\bm n| = 6,\Delta(\bm n)=1
		\atop
		|\bm \alpha| =  0, \bm \beta \neq \bm \gamma
	    }
		R_1(\bm n)
		J^{\bm \alpha} q^{\bm \beta} \bar q^{\bm \gamma},
	\end{align}
with
	\begin{equation*}
		\left|Z_1(\bm n)\right|, \left|R_1({\bm n})\right|\leq \epsilon^2,
	\end{equation*}	
	which implies
	\begin{equation}\label{122502}
		\left|Z_1(\bm n)\right|
		,
		\left|R_1(\bm n)\right| 
		\le 
		\epsilon^{\frac12\left(\left|\bm n\right|-2\right)},
	\end{equation}
    where using $|\bm n| = 6$.

    \subsubsection{Homological equation}

     Following the standard approach,
     \begin{equation*}
     	\mathcal H_2=\mathcal H_1 \circ X_{F_1}^1,
     \end{equation*}
     where $X_{F_1}^1$ is the symplectic transformation obtained from the Hamiltonian function
     \begin{align}\label{F1}
     	{F}_1=
     	\sum_{
     		 |\bm n|=6
     		, \, \Delta(\bm n)=1\atop |\bm \alpha|=0
     	}F_1(\bm n)J^{\bm \alpha}q^{\bm \beta}\bar q^{\bm \gamma},
     \end{align}
     where
     \begin{equation*}
     	F_1(\bm n)=\frac{R_1(\bm n)}{\sqrt{-1} \sum\limits_{\bm j\in\mathbb{Z}^d}(\beta_{\bm j}-\gamma_{\bm j})\omega_{\bm j}}.
     \end{equation*}
     Since the frequency $\bm  \omega$ satisfies the $(\eta,M)$-nonresonant conditions (\ref{122106})
     , then by (\ref{122502}) one has
    \begin{align*}
    	\left|F_1(\bm n)\right|
    	\leq
    	\eta^{-1} \left(1+\left(\bm \beta-\bm\gamma\right)^-\right)^{3\sigma}
    	(2+10\Delta(\bm \beta-\bm\gamma))^{2d\left|\bm \beta-\bm\gamma\right|}
    	\epsilon^{\frac12\left(\left|\bm n\right|-2\right)}.
    \end{align*}
   Then using the facts that
    \begin{equation} \label{k}
    	\Delta(\bm \beta-\bm \gamma) \le \Delta(\bm n)
    	, \quad
    	|\bm \beta-\bm \gamma| \le |\bm n|
    	, \quad \text{and} \quad
     (\bm\beta-\bm \gamma)^- \le \bm n^+,
    \end{equation}
    one obtains
 \begin{align}\label{111901..1}
    	\left|F_1(\bm n)\right|
    	\leq 
    	\epsilon^{\frac12\left(\left|\bm n\right|-2\right)}\cdot
    	\eta^{-1} \left(1+\bm n^+\right)^{3\sigma}
    	(2+10\Delta(\bm n))^{2d\left|\bm n\right|}
    	.
    \end{align}

    Since
     \begin{align*}
     	\frac{\partial F_1}{\partial \bar{q}_{\bm j}}
     	&=
     \sum_{|\bm n|=6 , \Delta(\bm n)=1\atop |\bm \alpha|=0}
     	\gamma_{\bm j} F_1(\bm n)J^{\bm \alpha}q^{\bm \beta}\bar q^{\bm \gamma - e_{\bm j}},
     \end{align*}
    then using (\ref{111901..1}) and $|\bm n|=6$ we have
     \begin{align}\label{F1q}
     	\left| \frac{\partial F_1}{\partial \bar{q}_{\bm j}} \right|
     	& \le \epsilon^{1.8} \left(1+|\bm j|_1\right)^{-2 \sigma}
     \end{align}
     where noting that
     \begin{equation*}
     	0 \le \bm n^+ - |\bm j|_1   \le d \Delta(\bm n), \quad \bm j \in \operatorname{supp} \bm n.
     \end{equation*}
     Consequently, we get that for $|t| \le 1$, $X_{F_1}^t$ is close to an identity map.
     Further, from (\ref{J}) and (\ref{F1q}) we have
     \begin{align}
     	|J_{\bm j} \circ \Phi_{F_1}^1 |
     	 & \label{Jpsi1.1}
     	 \le (1+\epsilon^{0.5})
     	 \left(1+|\bm j|_1\right)^{-3 \sigma}.
     \end{align}
     In fact, we have
     \begin{align*}
     	\nonumber
     	|J_{\bm j} \circ \Phi_{F_1}^1 |
     	&=|\epsilon^{-1}(I_{\bm j} \circ \Phi_{F_1}^1  - \zeta_{\bm j}  )|
     	\\& \nonumber
     	\le \epsilon^{-1}
     	|I_{\bm j} \circ \Phi_{F_1}^1 -I_{\bm j}|
     	+ \epsilon^{-1}|I_{\bm j}  -\zeta_{\bm j}|
     	\\
     	&\le (1+\epsilon^{0.5})
     	\left(1+|\bm j|_1\right)^{-3 \sigma}
     	,
     \end{align*}
     where the last inequality is based on (\ref{J}) and (\ref{F1q}). 


   \subsubsection{Estimate of $\mathcal{H}_2$}
     Next, define $\left\{H,G\right\}^{(0)}=H$ and
	\begin{align*}
		\left\{H,G\right\}^{(l)}=\left\{
		\{ H , G \}^{(l-1)},G\right\},
		\quad l\geq 1.
	\end{align*}
Using Taylor's formula yields
     \begin{align*}
     	\mathcal{H}_2=&
     	\mathcal{H}_1 \circ X_{F_1}^{1}\\
     	=&D+  \{ D , F_1 \}+\frac{1}{2!} \{ D , F_1 \}^{(2)}+ \cdots
     	\\
     	&+\mathcal{J}+  \{ \mathcal{J} , F_1 \}+\frac{1}{2!} \{ \mathcal{J} , F_1 \}^{(2)}+ \cdots
     	\\
     	&+Z_1+\{ Z_1 , F_1 \} + \frac{1}{2!} \{ Z_1 , F_1 \}^{(2)}+ \cdots
     	\\
     	&+R_1+\{ R_1 , F_1 \} + \frac{1}{2!} \{ R_1 , F_1 \}^{(2)}+ \cdots
     	\\
     	\coloneqq& D+\mathcal{J} + Z_2 + R_2 +
     	 O\left(  \epsilon^{0.24M } \right),
     \end{align*}
     where
     \begin{equation*}
     	Z_2-Z_1=\sum_{
     		|\bm n| = 8, \, \Delta(\bm n) \le 1
     		\atop \bm \beta =\bm \gamma
     	}
     	Z_2(\bm n)
     	J^{\bm \alpha} q^{\bm \beta} \bar q^{\bm \gamma},
     \end{equation*}
        \begin{equation*}
        	R_2=\sum_{
        		8\leq |\bm n|\leq {M}
        		\atop
        		\Delta(\bm n)\leq {M/4}
        	}
        	R_2(\bm n)
        	J^{\bm \alpha} q^{\bm \beta} \bar q^{\bm \gamma},
        \end{equation*}
    and
    \begin{align*}
    	O\left(\epsilon^{0.24M}\right)=
    	\frac{1}{ M^* !}
    	\int_{0}^{1} \left( 1-t \right)^{ M^*  }
    	\{ D+\mathcal{J}+Z_2+R_2 , F_2 \}^{(M^*+1)} \circ X_{F_2}^t \ dt,
    \end{align*}
    with $M^*= [ M/4 ]$.

    Firstly, we will estimate
    $\left|Z_{2}(\bm n)\right|$,
    and
    $\left|R_{2}(\bm n)\right|$.
    It suffices to estimate
    \begin{align*}
    	\{ D , F_1 \}^{(l)}, \
    	\{ \mathcal{J} , F_1 \}^{(l)}, \
    	\{ Z_1 , F_1 \}^{(l)} \ \text{and} \
    	\{ R_1 , F_1 \}^{(l)}, \quad  1 \le l \le M^*.
    \end{align*}

    We start with $	\{ Z_1 , F_1 \}^{(l)}, \ l \ge 1 $.
    When $l=1$,
    	from Lemma \ref{pb} one has
    	\begin{align}
    		\nonumber
    		\left|\{ Z_1 , F_1 \}(\bm n)\right|
    		&\le
    		2
    		\epsilon^{-1}
    		\left(
    		2^{|\bm n|}
    		\left(\Delta(\bm \mu)+ \Delta(\bm m)\right)
    		\right)
    		\left( \left| \bm n \right|+2 \right)^{2}
    		|Z_1(\bm \mu)| |F_1(\bm m)|
    		\\ \nonumber
    		&\le
    		2
    		\epsilon^{-1}
    		\left(
    		2^{|\bm n|}
    		\left(\Delta(\bm \mu)+ \Delta(\bm m)\right)
    		\right)
    		\left( \left| \bm n \right|+2 \right)^{2}
    		\epsilon^{\frac12\left(\left|\bm \mu\right|-2\right)}
    		\epsilon^{\frac12\left(\left|\bm m\right|-2\right)}
    		\\ &\quad \cdot
    		\eta^{-1} \left(1+\bm m^+\right)^{3\sigma}
    		(2+10\Delta(\bm m))^{2d\left|\bm m\right|},
    		\label{z1f1-1}
    	\end{align}
        where the last inequality is based on (\ref{122502}) and (\ref{111901..1}).
    Note that $Z_1,F_1$ is given by (\ref{Z1}) and (\ref{F1}) separately. 
    	Then we have
    	\begin{equation}\label{230110172811}
    		\quad |\bm n| =10,\quad |\bm \alpha| = 0   
        \end{equation} 
    and
        \begin{equation}\label{230110172812}
        	\Delta(\bm n) \le \Delta(\bm \mu) +\Delta(\bm m)  = 2 =(|\bm n|-2)/4 ,
    	\end{equation}
    where the last inequality is based on (\ref{Dn}).
    Further, in view of (\ref{nmum}), (\ref{m+}) and (\ref{230110172812}) we can get that
    	 \begin{align}
    	 	\nonumber
    		(\ref{z1f1-1})
    		&\le  
    		\epsilon^{\frac12 \left(  |\bm n| -2 \right) -1 }
    		2^{|\bm n|-1}
    		(|\bm n| - 2)
    		\left( \left| \bm n \right|+2 \right)^{2}
    	\\&\quad \cdot
    		\eta^{-1} \left(1+\bm n^+\right)^{3 \sigma}
    		(1+d|\bm n|)^{3 \sigma}
    		(3|\bm n|)^{2d\left|\bm n\right|}.
    		\label{23013121}
    	\end{align}
    Further, from (\ref{230110172811}) and the fact $\sigma \ge d+1$, 
    we can simplify (\ref{23013121}) as
    \begin{align}\label{2301101743}
    	\left|\{ Z_1 , F_1 \}(\bm n)\right|    	
    	&
    	\le \frac13
    	(\epsilon/ \eta)^{1+(|\bm n| -2 ) /4}
    	(6d|\bm n|)^{4 \sigma |\bm n|}
    	(1+\bm n^+)^{ \sigma \left(  |\bm n|-6+|\bm \alpha|\right)}.
    \end{align}

    For $l \ge 2$, using Lemma \ref{pb}, Remark \ref{n+} and following the proof of (\ref{2301101743}), by induction we can get
    \begin{align}
    	\left|\{ Z_1 , F_1 \}^{(l)}(\bm n)\right|
        \nonumber
    	&\le
    		\epsilon^{2+l}
    	2^{2l-1}
    	(2^{|\bm n|-2} |\bm n| )^{l}
    	(|\bm n| + 2l)^{2l}
    \\ \nonumber &\quad \cdot
    	\left[
    	\eta^{-1} \left(1+\bm n^+\right)^{3\sigma}
    	(1+d|\bm n|)^{3\sigma}
    	(3|\bm n|)^{2d\left|\bm n\right|}
    	\right]^{l}
    	\\ \nonumber
    	 &\le
    	 \frac13
    	 	(\epsilon / \eta )^{2 +l}
    	(6d|\bm n|)^{4\sigma|\bm n|l}
    	 (1+{\bm n}^+)^{3\sigma l}
    	\\
    	&\le
    	\frac13
    		\left(  \epsilon / \eta \right)^{1 +(|\bm n|-2)/4}
        (6d|\bm n|)^{4 \sigma |\bm n|(|\bm n|-4)}
        (1+\bm n^+)^{ \sigma \left(  |\bm n|-6+|\bm \alpha|\right)}
    	,
    	\label{ZFL} 
    \end{align}
     where the last inequality is based on the fact that
     \begin{equation*}
     	l = (|\bm n|-6)/4 .
     \end{equation*}

   Similarly, for $l \ge 1$ one has 
   \begin{align*}
   	\left|\{ R_1 , F_1 \}^{(l)}(\bm n)\right|
   	&\le
   	\frac13
   		\left(  \epsilon / \eta \right)^{1+(|\bm n|-2)/4}
   	(6d|\bm n|)^{4 \sigma |\bm n|(|\bm n|-4)}
   	  (1+\bm n^+)^{ \sigma \left(  |\bm n|-6+|\bm \alpha|\right)}
   	,
   \end{align*}
   and 
   \begin{align*}
   	\left|\{ D , F_1 \}^{(l)}(\bm n)\right|
   	&\le
   	\frac13
   	\left(  \epsilon / \eta \right)^{1+(|\bm n|-2)/4}
   	(6d|\bm n|)^{4 \sigma |\bm n|(|\bm n|-4)}
   	  (1+\bm n^+)^{ \sigma \left(  |\bm n|-6+|\bm \alpha|\right)}
   	,
   \end{align*}
   where using $\{ D , F_1 \}=-R_1$.

    Now we estimate $	\{ \mathcal{J} , F_1 \}^{(l)}, \ l \ge 1 $. When $l=1$, we have $|\bm n|=8$ and $|\bm \alpha| = 1$. Then following the proof of (\ref{2301101743}), one has
    \begin{align*}
    	\left|\{ \mathcal{J} , F_1 \}(\bm n)\right|
    	&\le
    	\frac13
    		\left(  \epsilon / \eta \right)^{1+|\bm n|/4}
    	(6d|\bm n|)^{4 \sigma |\bm n|}
    	 (1+\bm n^+)^{ \sigma \left(  |\bm n|-6+|\bm \alpha|\right)}.
    \end{align*}
     When $l \ge 2$, we have $l=(|\bm n |-4)/4$, and $|\bm n| \ge 12$, which implies $3l \le |\bm n| -6$. Then following the proof of (\ref{ZFL}), one has
     \begin{align*}
     	\left|\{ \mathcal{J} , F_1 \}^{(l)}(\bm n)\right|
     	&\le
     	\frac13
     		\left(  \epsilon / \eta \right)^{2+l}
     	(6d|\bm n|)^{4 \sigma |\bm n|l}
     	  (1+\bm n^+)^{3 \sigma l}
     	\\
     	&\le
     	\frac13
     		\left(  \epsilon / \eta \right)^{1+|\bm n|/4}
     	(6d|\bm n|)^{4 \sigma |\bm n|(|\bm n|-4)}
     	 (1+\bm n^+)^{ \sigma \left(  |\bm n|-6+|\bm \alpha|\right)}.
     \end{align*}

     In conclusion, we have that
     \begin{align}\label{ZR2}
     	\left| Z_2{(\bm n)} \right| + \left| R_2{(\bm n)} \right|
     	&\le
     		\left(  \epsilon / \eta \right)^{1+(|\bm n|-2)/4}
     	(6d|\bm n|)^{4 \sigma |\bm n|(|\bm n|-4)}
     	  (1+\bm n^+)^{ \sigma \left(  |\bm n|-6+|\bm \alpha|\right)}
     	.
     \end{align}
     
     Finally, we have
    \begin{align*}
    	\left|\left\{ I_{\bm j} , O\left(\epsilon^{0.24M}\right) \right\}\right| \le  \epsilon^{0.24M} (6dM)^{4 \sigma M^2} (1+|\bm j|_1)^{-6 \sigma}
    	.
    \end{align*}
    The detail of the proof will be given in the next subsection (cf. (\ref{IW2})).

\subsection{The Second Step}
    Following the standard approach,
    \begin{equation*}
    	\mathcal H_3=\mathcal H_2 \circ X_{F_2}^1,
    \end{equation*}
    where $X_{F_2}^1$ is the symplectic transformation obtained from the Hamiltonian function
    \begin{align*}
    	{F}_2=
    	\sum_{
    		|\bm n|=8, \, \Delta(\bm n) \le 1
    		\atop |\bm \alpha| \le 1
    	}F_2(\bm n)J^{\bm \alpha}q^{\bm \beta}\bar q^{\bm \gamma},
    \end{align*}
    where
    \begin{equation*}
    	F_2(\bm n)=\frac{R_2(\bm n)}{\sqrt{-1} \sum\limits_{\bm j\in\mathbb{Z}^d}(\beta_{\bm j}-\gamma_{\bm j})\omega_{\bm j}}.
    \end{equation*}
    Since the frequency $\bm \omega$ satisfies the $(\eta,M)$-nonresonant conditions (\ref{122106}),
     then we have
    \begin{align*}
    	\nonumber
    	\left|F_2(\bm n)\right|
    	&\leq
    		\left(  \epsilon / \eta \right)^{1+(|\bm n|-2)/4}
    	(6d|\bm n|)^{4\sigma|\bm n|(|\bm n|-4)}
    	  (1+\bm n^+)^{ \sigma \left(  |\bm n|-6+|\bm \alpha|\right)}
    	\\ & \quad \cdot  	
    	\eta^{-1} \left(1+\bm n^+\right)^{3\sigma}
    	(2+10\Delta(\bm n))^{2d\left|\bm n\right|}
    	,
    \end{align*}
    where we have used (\ref{k}).
    Furthermore, one has
    \begin{align}
    	\frac{\partial F_2}{\partial \bar{q}_{\bm j}}
    	&=
    	\sum_{|\bm n|=8,\, \Delta(\bm n) \le 1
    	\atop |\bm \alpha| \le 1}
    \left(
    		     	\alpha_{\bm j} F_2(\bm n)J^{\bm \alpha - e_{\bm j}}q^{\bm \beta + e_{\bm j}}\bar q^{\bm \gamma }
    		     	+
    	\gamma_{\bm j} F_2(\bm n)J^{\bm \alpha}q^{\bm \beta}\bar q^{\bm \gamma - e_{\bm j}}
    	\right).
    	\label{F2q}
    \end{align}
    In view of (\ref{J}) and following the proof of (\ref{F1q}), we can conclude that the second term in the right-hand of (\ref{F2q}) can be estimated by
    \begin{align*}
    	\frac{1}{2}\epsilon^{2}(1+|\bm j|_1)^{-2\sigma},
    \end{align*}
    and the first term in the right-hand of (\ref{F2q}) can be estimated by
    \begin{align*}
    	\frac12 \epsilon^{2}(1+|\bm j|_1)^{-\sigma}.
    \end{align*}
    
     Consequently, we have
    \begin{align}\label{Fq2}
    	\left| \frac{\partial F_2}{\partial \bar{q}_{\bm j}} \right|
    	& \le 
    	\epsilon^{2} \left(1+|\bm j|_1\right)^{-\sigma}
    \end{align}
and
    \begin{equation*} 
    	|J_{\bm j} \circ \Phi_{F_1}^1 |
    	\le \left(1+\sum_{h=1}^{2}\epsilon^{0.5h}\right)
    	\left(1+|\bm j|_1\right)^{-2\sigma},
    \end{equation*}
    where the last inequality is based on (\ref{Fq2}) and following the proof of (\ref{Jpsi1.1}).
    
    In fact, we can get a better estimate of $|J_{\bm j} \circ \Phi_{F_1}^1 |$.
    Note that
    \begin{align}
    	\nonumber
    	\frac{d}{dt} I_{\bm j} \circ \Phi_{F_2}^t
    	&=\{ I_{\bm j} , F_2 \} \circ \Phi_{F_2}^t
    	\\
    	\nonumber
    	&=
    	\sum_{
    		|\bm n|=8, \, \Delta(\bm n) \le 1
    		\atop |\bm \alpha| \le 1
    	}F_2(\bm n)
    	\left\{ q_{\bm j} \bar{q}_{\bm j}
    	, J^{\bm \alpha}q^{\bm \beta}\bar q^{\bm \gamma} \right\} \circ \Phi_{F_2}^t \\
    	\label{dIF2}
    	&=
    	\sum_{
    		|\bm n|=8, \, \Delta(\bm n) \le 1
    		\atop |\bm \alpha|\le 1
    	}F_2(\bm n)
    	(\gamma_{\bm j}  -\beta_{\bm j})
    	J^{\bm \alpha}q^{\bm \beta}\bar q^{\bm \gamma}
    	\circ \Phi_{F_2}^t
    	.
    \end{align}
       From (\ref{dIF2}), we can conclude that
        \begin{align}\label{IF2}
        	|I_{\bm j} \circ \Phi_{F_2}^1 -I_{\bm j}|
        	&\le \epsilon^{2} (1+|\bm j|_1)^{-3\sigma}.
        \end{align}
        Then we have
        \begin{align}
        	\nonumber
        	|J_{\bm j} \circ \Phi_{F_2}^1 |
        	&=|\epsilon^{-1}(I_{\bm j} \circ \Phi_{F_2}^1  -  \zeta_{\bm j}  )|
        	\\\nonumber &
        	\le \epsilon^{-1}
        	 |I_{\bm j} \circ \Phi_{F_2}^1 -I_{\bm j}|
        	+ |J_{\bm j}|
        	\\ \label{Jpsi2}&
        	\le \left(1+\sum_{h=1}^{2}\epsilon^{0.5h}\right)(1+|\bm j|_1)^{-3\sigma},
        \end{align}
        where the last inequality is based on (\ref{Jpsi1.1}) and (\ref{IF2}).

    Next, using Taylor's formula yields
    \begin{align*}
    	\mathcal{H}_3=&
    	\mathcal{H}_2 \circ X_{F_2}^{1}\\
    	=&D+  \{ D , F_2 \}+\frac{1}{2!} \{ D , F_2 \}^{(2)}+ \cdots
    	\\
    	&+\mathcal{J}+  \{ \mathcal{J} , F_2 \}+\frac{1}{2!} \{ \mathcal{J} , F_2 \}^{(2)}+ \cdots
    	\\
    	&+Z_2+\{ Z_2 , F_2 \} + \frac{1}{2!} \{ Z_2 , F_2 \}^{(2)}+ \cdots
    	\\
    	&+R_2+\{ R_2 , F_2 \} + \frac{1}{2!} \{ R_2 , F_2 \}^{(2)}+ \cdots
    	\\
    	\coloneqq& D+\mathcal{J} + Z_3 + R_3 +  O\left( \epsilon^{0.24M} \right),
    \end{align*}
    where
    \begin{equation*}
    	Z_3-Z_2=\sum_{
    		|\bm n| = 10
    		, \,
    		\bm \beta=\bm\gamma
    	}
    	Z_3(\bm n)
    	J^{\bm \alpha} q^{\bm \beta} \bar q^{\bm \gamma},
    \end{equation*}
    \begin{equation*}
    	R_3=\sum_{
    		10\leq |\bm n|\leq M
    		\atop
    		\Delta(\bm n)\leq M/4
    	}
    	R_3(\bm n)
    	J^{\bm \alpha} q^{\bm \beta} \bar q^{\bm \gamma},
    \end{equation*}
    and
    \begin{align*}
    	O\left(\epsilon^{0.24M}\right)=
    	\frac{1}{ M^* !}
    	\int_{0}^{1} \left( 1-t \right)^{ M^*  }
    	\{ D+\mathcal{J}+Z_2+R_2 , F_2 \}^{(M^*+1)} \circ X_{F_2}^t \ dt,
    \end{align*}
    with $M^*= [ M/4 ]$.

    Firstly, we estimate
    $\left|Z_{3}(\bm n)\right|$
    and
    $\left|R_{3}(\bm n)\right|$.
    According to the proof of (\ref{ZR2}). It suffices to estimate
    $\{ Z_2,F_2 \}$. In fact, we have
    \begin{align*}
    	\left|\{ Z_2 , F_2 \}(\bm n)\right|
    	&\le
    	2 \epsilon^{-1}
    	\left(
    	2^{|\bm n|}
    	\left(|\bm n| -2 \right) /4
    	\right)
    	\left( \left| \bm n \right|+2 \right)^{2}
    	\\
    	& \quad \cdot
    		\left(  \epsilon / \eta \right)^{1+(|\bm \mu|-2)/4}
    	(6d|\bm \mu|)^{4\sigma|\bm \mu|(|\bm \mu|-4)}
    	  (1+\bm \mu^+)^{\sigma \left(  |\bm \mu| - 6 +|\bm{ \widehat\alpha}| \right) }
    	\\
    	& \quad \cdot
    		\left(  \epsilon / \eta \right)^{1+(|\bm m|-2)/4}
    	(6d|\bm m|)^{4\sigma|\bm m|(|\bm m|-4)}
    	 (1+\bm m^+)^{\sigma  \left( |\bm m| - 6 +|\widetilde {\bm \alpha }| \right)}
    	\\ & \quad \cdot  	
    	\eta^{-1} \left(1+\bm m^+\right)^{3\sigma}
    	(3|\bm m|)^{2d\left|\bm m\right|}
    	\\ & \le
    	\frac13
    	 (\epsilon/\eta)^{1+(|\bm n|-2)/4}
    		(6d|\bm n|)^{4\sigma|\bm n|}
    	(6d|\bm n|)^{4\sigma|\bm n|(|\bm n|-4-2)}
    	 (1+\bm n^+)^{ \sigma \left(  |\bm n|-6+|\bm \alpha|\right)}
    	\\
    	&\le
    	\frac{1}{3}
    	 (\epsilon/\eta)^{1+(|\bm n|-2)/4}
    	(6d|\bm n|)^{4\sigma|\bm n|(|\bm n|-4)}
    	 (1+\bm n^+)^{ \sigma \left(  |\bm n|-6+|\bm \alpha|\right)}
    	.
    \end{align*}
%
%
    Consequently, we have
    \begin{align}\label{Z3}
    	\left| Z_3{(\bm n)} \right| + \left| R_3{(\bm n)} \right|
    	&\le
    		(\epsilon/\eta)^{1+(|\bm n| - 2)/4}
    	(6d|\bm n|)^{4\sigma|\bm n|(|\bm n| -4)}
    	 (1+\bm n^+)^{ \sigma \left(  |\bm n|-6+|\bm \alpha|\right)}
    	.
    \end{align}

    Finally, in view of (\ref{dIF2}), we can conclude that
    \begin{align}
    	\label{IW2}
    	\left|\left\{ I_{\bm j} , O\left(\epsilon^{0.24M}\right) \right\}\right| \le  \epsilon^{0.24M} (6dM)^{4\sigma M^2} (1+|\bm j|_1)^{-6\sigma}.
    \end{align}

\subsection{The General Step}

\begin{lem}[\textbf{Iterative Lemma}]\label{62}
For $3 \le s \le (M-4)/2$, consider the Hamiltonian
\begin{eqnarray*}
\mathcal{H}_s=D+\mathcal{J}+Z_s+R_s+ O\left(\epsilon^{0.24M}\right),
\end{eqnarray*}
where 
\begin{equation*}
	Z_s-Z_{s-1}=\sum_{
		|\bm n| = 2s+4, 
		\atop \bm \beta =\bm \gamma
	}
	Z_s(\bm n)
	J^{\bm \alpha} q^{\bm \beta} \bar q^{\bm \gamma},
\end{equation*}
\begin{equation*}
	R_{s}=\sum_{
		2s+4\leq |\bm n|\leq M
		\atop
		\Delta(\bm n)\leq M/4
	}
	R_s(\bm n)
	J^{\bm \alpha} q^{\bm \beta} \bar q^{\bm \gamma},
\end{equation*}
and
\begin{align*}
	O\left(\epsilon^{0.24M}\right)=
	\frac{1}{ M^* !}
	\int_{0}^{1} \left( 1-t \right)^{ M^*  }
	\{ D+\mathcal{J}+Z_s+R_s , F_s \}^{(M^*+1)} \circ X_{F_s}^t \ dt,
\end{align*}
with $M^*= [ M/4 ]$.
Let $\bm \omega$ satisfy the $(\eta,M)$-nonresonant conditions \eqref{122106}
and we assume that
\begin{align}
	 \label{Zs}
	\left| Z_s (\bm n)\right| + \left| R_s (\bm n)\right|
	&\le
		(\epsilon/\eta)^{1+(|\bm n| - 2)/4}
	(6d|\bm n|)^{4\sigma|\bm n|(|\bm n| -4)}
	 (1+\bm n^+)^{ \sigma \left(  |\bm n|-6+|\bm \alpha|\right)}
	,
\end{align}
and
$ \Delta(\bm n) \le  (|\bm n|-2)/4$
in $Z_s+R_s$.

Then there exists a change of variables $X_{F_s}^1$ generated by the hamiltonian function $F_s$ satisfying 
\begin{align}\label{Fsq}
	\left| \frac{\partial F_s}{\partial \bar{q}_{\bm j}} \right|
	& \le \epsilon^{1+0.5s} \left(1+|\bm j|_1\right)^{-\sigma}
\end{align}
and
\begin{align}\label{Ipsis}
	|I_{\bm j} \circ \Phi_{F_s}^1 -I_{\bm j}|
	&\le \epsilon^{1+0.5s} (1+|\bm j|_1)^{-3\sigma},
\end{align}
such that
\begin{align*}
     \mathcal{H}_{s+1} &=\mathcal{H}_s\circ X_{F_s}^1
    =D+\mathcal{J}+Z_{s+1}+R_{s+1}
    + O\left(\epsilon^{0.24M}\right)
    ,
\end{align*}
where 
\begin{equation*}
	Z_{s+1}-Z_{s}=\sum_{
		|\bm n| = 2s+6, 
		\atop \bm \beta =\bm \gamma
	}
	Z_{s+1}(\bm n)
	J^{\bm \alpha} q^{\bm \beta} \bar q^{\bm \gamma},
\end{equation*}
and
\begin{equation*}
	R_{s+1}=\sum_{
		2s+6\leq |\bm n|\leq M
		\atop
		\Delta(\bm n)\leq M/4
	}
	R_{s+1}(\bm n)
	J^{\bm \alpha} q^{\bm \beta} \bar q^{\bm \gamma}.
\end{equation*}
Moreover, one has
\begin{align*}
	\left| Z_{s+1} (\bm n)\right| + \left| R_{s+1} (\bm n)\right|
	&\le
		(\epsilon/\eta)^{1+(|\bm n| - 2)/4}
	(6d|\bm n|)^{4\sigma|\bm n|(|\bm n| -4)}
	 (1+\bm n^+)^{ \sigma \left(  |\bm n|-6+|\bm \alpha|\right)}
	,
\end{align*}
and
$	\Delta(\bm n) \le (|\bm n|-2)/4 $
in $Z_{s+1}+R_{s+1}$.

\end{lem}

\begin{proof}
First, as a standard Birkhoff normal form technique, we know that
\begin{align*}
{F}_s=
\sum_{
	|\bm n|=2s+4
    }
  F_s(\bm n)J^{\bm \alpha}q^{\bm \beta}\bar q^{\bm \gamma},
\end{align*}
where
\begin{equation}\label{Fs}
F_s(\bm n)=\frac{R_s(\bm n)}{\sqrt{-1} \sum\limits_{\bm j\in\mathbb{Z}^d}(\beta_{\bm j}-\gamma_{\bm j}){\omega}_{\bm j}}.
\end{equation}
Since the frequency $ \bm \omega$ satisfies the $(\eta,M)$-nonresonant conditions (\ref{122106}),
 combining (\ref{Fs}) we get
\begin{align}
	\nonumber
\left|F_s(\bm n)\right|
&\leq
	\left(  \epsilon / \eta \right)^{1+(|\bm n|-2)/4}
(6d|\bm n|)^{4\sigma|\bm n|(|\bm n|-4)}
 (1+\bm n^+)^{ \sigma \left(  |\bm n|-6+|\bm \alpha|\right)}
\\ \label{Fsn}
& \quad \cdot  	
\eta^{-1} \left(1+\bm n^+\right)^{3\sigma}
(2+10\Delta(\bm n))^{2d\left|\bm n\right|}
\end{align}
where we have used (\ref{k}).
From (\ref{Fsn}) and following the proof of (\ref{Fq2}) in the second step, we have (\ref{Fsq}).
Furthermore, following the proof of (\ref{IF2}) we have (\ref{Ipsis}).

Next, using Taylor's formula yields
    \begin{align*}
    	\mathcal{H}_{s+1}=&
    	\mathcal{H}_s \circ X_{F_s}^{1}
    	\\
    	=&D+  \{ D , F_s \}+\frac{1}{2!} \{ D , F_s \}^{(2)}+ \cdots
    	\\
    	&+\mathcal{J}+  \{ \mathcal{J} , F_s \}+\frac{1}{2!} \{ \mathcal{J} , F_s \}^{(2)}+ \cdots
    	\\
    	&+Z_s+\{ Z_s , F_s \} + \frac{1}{2!} \{ Z_s , F_s \}^{(2)}+ \cdots
    	\\
    	&+R_s+\{ R_s , F_s \} + \frac{1}{2!} \{ R_s , F_s \}^{(2)}+ \cdots
    	\\
    	\coloneqq & D+ \mathcal{J} +Z_{s+1} +R_{s+1} + O\left(\epsilon^{0.24M}\right),
    \end{align*}
    where
    $$Z_{s+1} - Z_{s} = \sum_{|\bm n|=2s+6 \atop \bm \beta = \bm \gamma} Z_{s+1}(\bm n) J^{\bm \alpha} q^{\bm \beta} \bar{q}^{\bm \gamma},$$
    and
    $$R_{s+1}=\sum_{ 2s+6 \le |\bm n| \le M \atop \Delta(\bm n) \le M/4} R_{s+1}(\bm n) J^{\bm \alpha} q^{\bm \beta} \bar{q}^{\bm \gamma}.$$

    In view of (\ref{Zs}) and following the estimate of (\ref{Z3}) and (\ref{IW2}) in the second step, we get
    \begin{align*}
    	\left| Z_{s+1}{(\bm n)} \right| + \left| R_{s+1}{(\bm n)} \right|
    	 \le 	(\epsilon/\eta)^{1+(|\bm n| - 2)/4}
    	(6d|\bm n|)^{4\sigma|\bm n|(|\bm n| -4)}
    	 (1+\bm n^+)^{ \sigma \left(  |\bm n|-6+|\bm \alpha|\right)}
    	,
    \end{align*}
    and
    \begin{align*}
    	\left|\left\{ I_{\bm j} , O\left(\epsilon^{0.24M}\right) \right\}\right| \le  \epsilon^{0.24M} (6dM)^{4\sigma M^2} (1+|\bm j|_1)^{-6\sigma}.
    \end{align*}
    Furthermore, in view of the second conclusion in Remark \ref{n+} (cf.(\ref{n})), we conclude by induction $ \Delta(\bm n) \le (|\bm n|-2)/4$  in $Z_{s+1}+R_{s+1}$.

\end{proof}
    \begin{rem}
    	Form (\ref{Ipsis}), and following the proof of (\ref{Jpsi2}) we can conclude that
    	\begin{align}\label{Jpsis}
    		|J_{\bm j} \circ \Phi_{F_s}^1 |
    		\le
    		\left( 1 + \sum_{h=1}^{s} \epsilon^{0.5h} \right) (1+|\bm j|_1)^{-3\sigma} ,
    	\end{align}
        which we have used during the iteration.
    \end{rem}

\section{Estimate on the measure}\label{Em}

    \begin{lem}
    	For the set $\Omega$ given by (\ref{122302}),
    	there exists a set $\Omega'
    	\subset \Omega
    	$ satisfying
    	\begin{equation}\label{122304}\operatorname{mes} \left(\Omega'\right) < \eta,
    \end{equation}
    	such that for any $\zeta \in  \Omega \setminus \Omega' $, the frequency $\bm \omega$ given by (\ref{122303})
    	are $(\eta,M)$-nonresonant.
    \end{lem}
    \begin{proof}
    Define the set
    $\mathfrak{R}(\bm k)$ by
    	\begin{align*}
    		\left\{
    		\zeta\in\Omega
    		:\
    		\left|\sum_{\bm j\in\mathbb{Z}^d}
    		k_{\bm j}
           {\omega}_{\bm j}
    		\right|
    		<
    		\frac{\eta}
    		{ (1+\bm k^-)^{3\sigma}
    			(2+10\Delta(\bm k))^{2d\left|\bm k\right|} }
    		\right\},
    	\end{align*}
        where $\sigma\ge d+1$.
    	Let
    	\begin{equation*}
    		\Omega'=
    		\bigcup_{
    			 |\bm k|,\Delta(\bm k) \leq M}
    		\mathfrak{R}(\bm k),
    	\end{equation*}
    	and it is easy to see that for each $ \zeta \in  \Omega \setminus \Omega'$ the frequency  $\bm \omega$ are $(\eta,M)$-nonresonant.

    Now it suffices to prove the estimate (\ref{122304}) holds. Firstly, one has
    	\begin{align*}
    		{\rm mes}(\mathfrak{R}(\bm k)) <
    		\frac{\eta }
    		{(1+\bm k^-)^{\sigma} \left(2+10\Delta(\bm k)\right)^{2d\left|\bm k\right|}
    		}.
    	\end{align*}

        Secondly, for any given $\bm k$ satisfying $|\bm k| ,\Delta(\bm k) \le M$, there exists $\bm s \in \mathbb{Z}^d$ satisfying $|\bm s|_1=\bm k^-$ such that $\operatorname{supp} \bm k \subset \mathcal{B}(\bm s)$, where 
        $$\mathcal{B}(\bm s) = \left\{ \bm j \Big| |\bm j - \bm s|_\infty \leq M \right\}.$$
Note that the number of $\bm k$ satisfying $\operatorname{supp} \bm k \subset \mathcal{B}(\bm s)$ is no more than
    \begin{equation*}
    	\sum_{i=0}^{M} \sum_{j=6}^{M} 
    	(2i+2)^{dj},
    \end{equation*} 
    where we take $\Delta(\bm k)=i$, and $|\bm k|=j$.
	Then we have that,
    	\begin{equation*}
    		{\rm mes} \left(\Omega'\right)  \le
    		\sum_{\bm s \in \mathbb{Z}^d}
    		\sum_{i=0}^{M} \sum_{j=6}^{M} 
    		(2i+2)^{dj}
    		\frac{\eta}
    		{(1+|\bm s|_1)^{\sigma} (2+10i)^{2dj}
    		}
    		<\eta.
    	\end{equation*}
    \end{proof}

\section{The proof of main theorem}\label{mainthm}
\subsection{Proof of Theorem \ref{main}}
\begin{proof}
	From Lemma \ref{62} (Iterative Lemma), we get that when $s=(M-4)/2$, for the given Hamiltonian function $\mathcal{H}_1$  there exists a change of variables $\Phi=X_{F_1}^1 \circ X_{F_2}^1 \circ \cdots \circ X_{F_{(M-4)/2}}^1$ such that
	\begin{align*}
		\widetilde{H}\left(\widetilde{q},\widetilde{\bar{q}}\right) &=\mathcal{H}_1\circ \Phi
		=D+\mathcal{J}+O(|\widetilde{I}|)
		+ O\left(\epsilon^{0.24M}\right)
		,
	\end{align*}
    where $O(|\widetilde I|)$ means the terms depending on action variables $\widetilde I=(\widetilde{I}_{\bm j})_{\bm j\in\mathbb{Z}^d}$ with $\widetilde I_{\bm j}=\left|\widetilde{q}_{\bm j}\right|^2$ only.
    In view of (\ref{Fsq}), we can conclude that $\Phi$ is close to identity map.
    Furthermore, from (\ref{Jpsis}) we get that
    \begin{align*}
    	\left|\widetilde{J}_{\bm j} \coloneqq J_{\bm j} \circ \Phi \right|
    	\le
    	2 (1+|\bm j|_1)^{-3\sigma} .
    \end{align*}
    Let 
    \begin{equation}\label{M}
    	M= \frac{\ln \epsilon^{-1}}{100\sigma\ln \ln \epsilon^{-1}} .
    \end{equation}
    Then
    \begin{equation*}
    	O\left(\epsilon^{0.24M}\right)
    	\sim
    	O\left(
    	\epsilon
    	\exp\left\{\frac{-\left|\ln \epsilon\right|^2}{10^4 \sigma \ln\ln\epsilon^{-1}}\right\}\right).
    \end{equation*}
	
	 Firstly, we prove the long time stability in the new coordinate system, i.e. the stability time
     \begin{align}\label{t}
     	T^*\ge \exp\left\{\frac{\left|\ln \epsilon\right|^2}{10^4 \sigma \ln\ln\epsilon^{-1}}\right\},
     \end{align}
 where
\begin{align}\label{T*}
	T^*=\inf \left\{|t| \, \Bigg| \, 	 
	\left| \widetilde{J}_{\bm j} (t) \right| 
	= 
	4 (1+|\bm j|_1)^{-3\sigma}
	,\
	\exists \, \bm j \in \mathbb{Z}^d
	\right\}.
\end{align}

    Let 
    \begin{equation*}
    	\left| \widetilde{q}_{\bm j}(0) \right|^2 = \zeta_{\bm j}, 
    \end{equation*}
which implies that $ \left\| \widetilde{q}_{\bm j}(0) \right\|_{\sigma}^2 \le 1 $.
    Then we state that 
    \begin{equation}\label{qt}
    	\left\|  \widetilde{q}(t) \right\|_{\sigma}^2 \le 4, \quad \text{when} \quad |t| \le T^*.
    \end{equation}
    Otherwise, there exists $t^*$ satisfying $|t^*| \le T^*$ such that $\left\|  \widetilde{q}(t^*) \right\|_{\sigma}^2 > 4$ (i.e. there exists $\bm j_0 \in \mathbb{Z}^d$ such that $|\widetilde{q}_{\bm j_0}(t^*)|^2 > 4 (1+|\bm j_0|_1)^{-2\sigma}$). Note that
    \begin{equation*}
    	\epsilon \widetilde{J}_{\bm j} = \left|\widetilde{q}_{\bm j}\right|^2 - \zeta_{\bm j},
    \end{equation*}
    thus we have that
    \begin{align*}
    		\left| \widetilde{J}_{\bm j_0} (t^*) \right|
    		&\ge \epsilon^{-1} \left( |\widetilde{q}_{\bm j_0}(t^*)|^2 - \zeta_{\bm j_0} \right)
    		\ge 3 \epsilon^{-1} (1+|\bm j_0|_1)^{-2\sigma},
    \end{align*}
    which is in contradiction to (\ref{T*}).
    
    In view of (\ref{T*}) and (\ref{qt}), we conclude that
    for any $\bm j \in \mathbb{Z}^d$ and $|t| \le T^*$,
\begin{equation*}
	\left|  \frac{d}{dt}
	|\widetilde{q}_{\bm j}(t)|^2 \right|
	=\left|  \left\{ \widetilde I_{\bm j} , \widetilde{H} \right\} \right|
	<
	\epsilon
    \exp\left\{ \frac{-|\ln \epsilon|^2}{10^4 \sigma  \ln \ln \epsilon^{-1}} \right\}
	 (1+|\bm j|_1)^{-6\sigma}.
\end{equation*}
     If (\ref{t}) does not hold, then by using Newton-Leibiniz formula, one has
    \begin{align*}
    	 \left| 
    	 \left| \widetilde q_{\bm j}(T^*) \right|^2 
    	 -  \zeta_{\bm j} 
    	 \right|
    	 &\le 
    	T^*
    	\epsilon
    	 \exp\left\{ \frac{-|\ln \epsilon|^2}{10^4 \sigma \ln \ln \epsilon^{-1}} \right\} (1+|\bm j|_1)^{-6\sigma}
    	 \\&<
    	 \epsilon
    	 (1+|\bm j|_1)^{-6\sigma}
    	 ,
    \end{align*}
    which implies
    \begin{align}\label{012935}
    	\left| \widetilde J_{\bm j}(T^*) \right|
    	 = 
    	 \epsilon^{-1}
    	\left| 
    	\left| \widetilde q_{\bm j}(T^*) \right|^2 -  \zeta_{\bm j} 
    	\right|
    	<
    	(1+|\bm j|_1)^{-6\sigma}
    	.
    \end{align}
    On the other hand, in view of (\ref{T*}), we get that there exists $\bm j'_0$ such that 
    \begin{equation*}
    	\left| \widetilde J_{\bm j'_0}(T^*) \right|=4(1+|\bm j'_0|_1)^{-3\sigma},
    \end{equation*}
    which is in contradiction to (\ref{012935}). 
    
    Finally, note that $\Phi$ is close to identity map, we can finish the proof by coming back to the ordinary coordinate system including the stretching transformation (\ref{ss}). 
\end{proof}

\subsection{Proof of Theorem \ref{th2}}
\begin{proof}
	Firstly, we introduce a new phase space, which is defined by 
	\begin{equation}\label{well}
		\widetilde{\ell}_{\infty}^{\sigma}=\left\{q=(q_{\bm j})_{\bm j\in\mathbb{Z}^d}:
		\interleave  q  \interleave_{\sigma}
		:=
		\sup_{\bm j \in \mathbb{Z}^d}  |q_{\bm j}|(1+\langle\bm j\rangle)^\sigma
		<\infty\right\},
	\end{equation}
    with $\langle\bm j\rangle = \max\{ |\bm j|_1 , 1 \}$.
    Then condition (\ref{cd2}) implies that
   \begin{equation*}
   	\interleave  q(0)  \interleave_{\sigma} \le \epsilon.
   \end{equation*}
    
    Secondly, by using (\ref{122301}) and (\ref{010901}) we have that for any $\bm j \in \mathbb{Z}^d$
    \begin{align}\label{IH}
    	\left|  \left\{ I_{\bm j} , H \right\} \right|
    	&= 
    	\left|  \left\{ I_{\bm j} , R \right\} \right|
    	\le 
    	\sum_{\bm\beta,\bm\gamma\in\mathbb{N}^{\mathbb{Z}^d},|\bm \beta+\bm \gamma|=6\atop
    		\Delta( \bm \beta+\bm\gamma)\leq 1}
    	\left| R^{\beta\gamma}  \right|
    	\left|   \gamma_{\bm j} - \beta_{\bm j}  \right|
    	\left|  q^{\beta}\bar q^{\gamma} \right| .
    \end{align}
    It suffices to consider the term 
    \begin{equation*}
    	\left\{ I_{\bm j} , q^{\bm\beta} \bar{q}^{\bm\gamma} \right\} \neq 0,
    \end{equation*}  
    which implies that 
    \begin{equation}\label{013043}
    	\bm j \in \operatorname{supp} \left( \bm\beta + \bm\gamma \right).
    \end{equation} 
    Then if we assume that $\interleave q(t) \interleave_\sigma \le 2\epsilon$, by using (\ref{013043}) and the facts that 
    \begin{equation*}
    	|\bm \beta+\bm \gamma|=6
    	\quad \text{and} \quad
    	\Delta( \bm \beta+\bm\gamma)\leq 1,
    \end{equation*}
    one has 
    \begin{equation}\label{013049}
    	\left| q^{\bm\beta} \bar{q}^{\bm\gamma} \right| \le 
    	\left(2\epsilon \right)^6 
    	\left( 1+ \langle \bm j \rangle \right)^{-\sigma}
    	\left( \max\left\{ \langle \bm j \rangle,2 \right\} \right)^{-5\sigma}.
    \end{equation}
    Thus, in view of (\ref{epd}), (\ref{IH}) and (\ref{013049}), noting that $|R^{\bm \beta \bm \gamma}|\le 1$ (cf. (\ref{010901})) we get	
    \begin{align}
    	\nonumber
    	\left|  \left\{ I_{\bm j} , H \right\} \right|
    	&\le
    	4^{6d} \cdot 6
    	(2\epsilon)^6 (1+\langle\bm j\rangle)^{-3\sigma} 
    	\left( 
    	\frac{1+\langle\bm j\rangle}
    	{\max\left\{ \langle \bm j \rangle,2 \right\}}
    	\right)^{2\sigma}
    	2^{-3\sigma}
    	\\ & \le \epsilon^{5} (1+\langle \bm j \rangle )^{-3\sigma}
    	2^{-\sigma}
    	,
    	\label{IH2}
    \end{align}
    where the last inequality is based on the fact 
    \begin{equation*}
    	\frac{1+\langle\bm j\rangle}{\max\left\{ \langle \bm j \rangle,2 \right\}} \le \frac{3}{2}.
    \end{equation*}
    
    Finally, following the proof of (\ref{t}), we can finish the proof of (\ref{1.2}).
    Furthermore, note that $\sigma$ can be chosen free from $\epsilon$, we can finish the proof of (\ref{th2.2}) by taking $\sigma\geq \epsilon^{-1}$.
\end{proof}

\section*{Acknowledgments}
   H.C.  was supported by NNSF of China  (No. 11671066, 11401041) and NSFSP (No. ZR2019MA062). Y.S.  was supported by NNSF of China  (No.  12271380).


\begin{thebibliography}{SBFS07}

\bibitem[And58]{And58}
P.W. Anderson.
\newblock Absence of diffusion in certain random lattices.
\newblock {\em Physical review}, 109(5):1492--1505, 1958.

\bibitem[BFG88]{BFG88}
G.~Benettin, J.~Fr\"{o}hlich, and A.~Giorgilli.
\newblock A {N}ekhoroshev-type theorem for {H}amiltonian systems with
  infinitely many degrees of freedom.
\newblock {\em Comm. Math. Phys.}, 119(1):95--108, 1988.

\bibitem[Bou04]{Bou04}
J.~Bourgain.
\newblock Remarks on stability and diffusion in high-dimensional {H}amiltonian
  systems and partial differential equations.
\newblock {\em Ergodic Theory Dynam. Systems}, 24(5):1331--1357, 2004.

\bibitem[BW07]{BW07}
J.~Bourgain and W.-M. Wang.
\newblock Diffusion bound for a nonlinear {S}chr\"{o}dinger equation.
\newblock In {\em Mathematical aspects of nonlinear dispersive equations},
  volume 163 of {\em Ann. of Math. Stud.}, pages 21--42. Princeton Univ. Press,
  Princeton, NJ, 2007.

\bibitem[BW08]{BW08}
J.~Bourgain and W.-M. Wang.
\newblock Quasi-periodic solutions of nonlinear random {S}chr\"{o}dinger
  equations.
\newblock {\em J. Eur. Math. Soc. (JEMS)}, 10(1):1--45, 2008.

\bibitem[CMS22]{CMS22}
H.~Cong, L.~Mi, and Y.~Shi.
\newblock Super-exponential stability estimate for the nonlinear
  {S}chr\"{o}dinger equation.
\newblock {\em J. Funct. Anal.}, 283(12):Paper No. 109682, 24, 2022.

\bibitem[CSZ21]{CSZ21}
H.~Cong, Y.~Shi, and Z.~Zhang.
\newblock Long-time {A}nderson localization for the nonlinear {S}chr\"{o}dinger
  equation revisited.
\newblock {\em J. Stat. Phys.}, 182(1):Paper No. 10, 22, 2021.

\bibitem[FKS09]{FKS09}
S.~Fishman, Y.~Krivolapov, and A.~Soffer.
\newblock Perturbation theory for the nonlinear {S}chr\"{o}dinger equation with
  a random potential.
\newblock {\em Nonlinearity}, 22(12):2861--2887, 2009.

\bibitem[FKS12]{FKS12}
S.~Fishman, Y.~Krivolapov, and A.~Soffer.
\newblock The nonlinear {S}chr\"{o}dinger equation with a random potential:
  results and puzzles.
\newblock {\em Nonlinearity}, 25(4):R53--R72, 2012.

\bibitem[FSW86]{FSW86}
J.~Fr\"{o}hlich, T.~Spencer, and C.~E. Wayne.
\newblock Localization in disordered, nonlinear dynamical systems.
\newblock {\em J. Stat. Phys.}, 42(3-4):247--274, 1986.

\bibitem[SBFS07]{SBFS07}
T.~Schwartz, G.~Bartal, S.~Fishman, and M.~Segev.
\newblock Transport and {A}nderson localization in disordered two-dimensional
  photonic lattices.
\newblock {\em Nature}, 446(7131):52--55, 2007.

\bibitem[Sha07]{Sha07}
B.~Shapiro.
\newblock Expansion of a {B}ose-{}einstein condensate in the presence of
  disorder.
\newblock {\em Phys. Rev. Lett.}, 99(6):060602,1--4, 2007.

\bibitem[WZ09]{WZ09}
W.-M. Wang and Z.~Zhang.
\newblock Long time {A}nderson localization for the nonlinear random
  {S}chr\"{o}dinger equation.
\newblock {\em J. Stat. Phys.}, 134(5-6):953--968, 2009.

\end{thebibliography}

\end{document}